\documentclass[11pt]{article}
\usepackage[tbtags]{amsmath}
\usepackage{amssymb}
\usepackage{amsthm}
\usepackage[misc]{ifsym}
\usepackage{cases}
\usepackage{cite}	
\usepackage{mathrsfs}	
\usepackage{xcolor}	
\usepackage{graphicx}	
\usepackage{float}	
\usepackage{color}	
\usepackage{hyperref}	
\usepackage[ruled,linesnumbered]{algorithm2e}
\usepackage{amsmath,amssymb}

\hypersetup{hypertex=true,
	colorlinks=true,
	linkcolor=red,
	anchorcolor=blue,
	citecolor=blue}

\numberwithin{equation}{section}
\normalsize

\title{\bf Stochastic Optimal Control Problem under Inside Information\thanks{This work is supported by National Natural Science Foundations of China (12471419, 12271304), and Shandong Provincial Natural Science Foundation (ZR2024ZD35).}}
\author{\normalsize
	Yuxin Ye\thanks{\textit{School of Mathematics, Shandong University, Jinan 250100, P.R.  China, E-mail: yeyuxin@mail.sdu.edu.cn}},
	\ Jingtao Shi\thanks{\textit{Corresponding author, School of Mathematics, Shandong University, Jinan 250100, P.R. China, E-mail: shijingtao@sdu.edu.cn}}}	
\date{}

\newtheorem{mypro}{Proposition}[section]
\newtheorem{mythm}{Theorem}[section]
\newtheorem{mydef}{Definition}[section]
\newtheorem{mylem}{Lemma}[section]
\newtheorem{Remark}{Remark}[section]
\newtheorem{Corollary}{Corollary}[section]

\begin{document}
	
\maketitle

\noindent{\bf Abstract:}\quad
	This paper is concerned with a stochastic optimal control problem under inside information. The control process depends on an $\mathcal{F}_{T_0}$-measurable random variable $Y$, representing the static inside information, and is adapted to the enlarged filtration generated by the underlying Brownian motion and the random variable $Y$. Accordingly, the traditional stochastic integral fails to be well-defined in this non-adapted setting; we adopt forward integrals to formulate the stochastic integral terms in the system. By means of the Donsker delta function, the original controlled system is transformed into a $y$-parameterized system. We further establish the existence of solutions to forward \textit{stochastic differential equations} (SDEs), and prove the uniqueness of solutions via flow transformation techniques. Under a Gaussian assumption on $Y$, we derive both necessary and sufficient optimality conditions for the aforementioned control problem. Subsequently, we formulate the \textit{linear-quadratic} (LQ) optimal control problem under inside information. Through the $y$-parameterized transformation, the original problem is converted into an LQ control problem with random coefficients. A numerical example for the LQ case is provided at the end to validate our theoretical findings.
	\vspace{2mm} 
	
\noindent{\bf Keywords:}\quad Stochastic optimal control, inside information, anticipative stochastic calculus, linear-quadratic control, random coefficient 
	
	\vspace{2mm}
	
\noindent{\bf Mathematics Subject Classification:}\quad 93E20, 60H10, 60H05
	
\section{Introduction}\label{sec1}

	Insider trading refers to transactions featuring information asymmetry in financial markets. The economic models of insider trading were first developed by Kyle \cite{Kyle-1985} and Back \cite{Back-1992} which have become the benchmark framework for subsequent research and are known as the Kyle–Back equilibrium model. Within this framework, the insider can observe or rationally anticipate certain private market information, and is therefore also called an informed trader. This private information typically takes the intrinsic value (liquidation value) of the underlying asset representing forward-looking information about future payoffs. Meanwhile, the market also includes uninformed noise traders, who lack such information and trade for liquidity purposes, as well as market makers, who determine asset prices based on aggregate order flows to clear the market. Noise traders serve as camouflage for insiders, hiding their trading behaviors from market makers. Accordingly, insider trading models are essentially asset pricing problems under information asymmetry.
	
	There is extensive related literature for insider trading. In Aase et al. \cite{Aase-Bjuland-Oksendal-2012}, the static inside information $\tilde{v}$ is treated as a signal process, while the aggregate order flow $y$ serves as an observation process. By combining linear filtering theory with the calculus of variations, the optimal trading strategy for insiders was derived. Danilova \cite{Danilova-2010}, Caldentey and Stacchetti \cite{Caldentey-Stacchetti-2010} generalize static terminal insider information to dynamic information processes. Biagini et al. \cite{Biagini-Hu-Meyer-Oksendal-2012} apply fractional Brownian motion noise to capture the memory effect exhibited by uninformed traders. Additional relevant works include, for instance, \cite{Back-Pedersen-1998}, \cite{Ma-Sun-Zhou-2018}. 
	
	Motivated by such frameworks, numerous studies have explored insider optimal control and investment problems, where investors possess advance information about the system, namely, insider information. More specifically, each investor is assumed to observe a random variable $Y$ that is $\mathcal{F}_T$- or $\mathcal{H}_t$-measurable, with the filtration inclusion $\mathcal{F}_t \subset \mathcal{H}_t$, from the initiation of trading. Under this enlarged filtration setting, standard SDEs fail to be well defined. Pikovsky and Karatzas \cite{Pikovsky-Karatzas-1996} take the terminal value of Brownian motion as inside information, i.e., $Y=W(1)$. They construct a $\mathbb{G}$-Brownian motion $\tilde{W}$ for the enlarged filtration $\mathcal{G} = \mathcal{F}_t \vee \sigma(Y)$ using a Brownian bridge, derive the optimal investment strategy and prove that $W$ constitutes a $\mathbb{G}$-semimartingale with an explicit semimartingale decomposition.
	
	Biagini and \O ksendal \cite{Biagini-Oksendal-2005} do not impose the a priori assumption that $B$ is an $\mathbb{H}$-semimartingale. Instead, they adopt a more general approach, namely, the forward integrals, originally proposed by Russo and Vallois \cite{Russo-Vallois-1993}, to characterize non-adapted stochastic integrals and study the insider's optimal investment problem. They show that if the optimal strategy $\pi^*$ exists and the forward integrand is non-zero, then $B$ becomes an $\mathcal{H}_t$-semimartingale. Their paper pioneered the use of forward integrals to define non-adapted stochastic integrals. Draouil and \O ksendal \cite{Draouil-Oksendal-2015} examine a general nonlinear insider optimal control problem with jumps. An $\mathcal F_{T_0}$-measurable random variable $Y$ represents inside information, and the control $u(\cdot,Y)$ is adapted to the enlarged filtration $\mathbb{H}=\mathbb{F}\vee\sigma(Y)$. All stochastic integral terms in the system are interpreted as forward integrals. Via the Donsker delta function, the forward state system is converted to a $y$-parameterized It\^o's type SDE, and both necessary and sufficient maximum principles for optimality are derived. Their subsequent work \cite{Draouil-Oksendal-2016} further investigates stochastic differential games under inside information.
	
	Besides forward integrals, non-adaptive stochastic integrals can also be formulated via Skorohod integrals. Escudero \cite{Escudero-2018} compares the expected terminal wealth generated under these two integrals. It is demonstrated that forward integrals admit a more straightforward economic interpretation for insider trading compared with Skorohod integrals. For further related studies, see, e.g., \cite{Biagini-Oksendal-2006}, \cite{Aase-Oksendal-2019}, \cite{Oksendal-Rose-2017}, \cite{Elizalde-Escudero-Ichiba-2025}, \cite{Peng-2022}. Meanwhile, Hu and \O ksendal \cite{Hu-Oksendal-2007} introduce a quadratic penalty term imposed by the market on large insider trading volumes, which takes the form $\mathbb{E}\big[-\int_0^T |\mathbb{Q}\pi(t)|^2dt\big]$. In addition, Biagini and \O ksendal \cite{Biagini-Oksendal-2006} investigate quadratic minimum-variance hedging strategies for informed  insiders. 
	
	Motivated by the above literature, this paper studies a stochastic optimal control problem under inside information. The private information is captured by an $\mathcal{F}_{T_0}$-measurable random variable, and forward integrals are adopted to represent all stochastic integral terms in the state system. Firstly, we establish the existence and uniqueness, with particular emphasis on uniqueness, of solutions to general one-dimensional nonlinear SDEs. Rewriting the original forward equation as a $y$-parameterized It\^o's type SDE yields only the existence of solutions, whereas uniqueness has remained an open issue in prior work. Secondly, we derive the necessary condition together with Arrow’s sufficient condition for optimality. We then turn to an LQ insider optimal control problem. Applying the Donsker delta function of $Y$, we convert the forward SDE into a $y$-parameterized It\^o's type SDE which reduces the original problem to an LQ optimal control problem with random coefficients. A \textit{backward stochastic differential equation} (BSDE) is introduced as the \textit{stochastic Riccati equation} (SRE) to solve for the optimal feedback control.
	
	The remainder of this paper is organized as follows. Section 2 presents the necessary theoretical preliminaries, formulates the problem, and establishes the existence and uniqueness of solutions to the state equation. Section 3 derives the necessary optimality condition as well as Arrow’s sufficient condition. Section 4 solves the LQ insider optimal control problem and a numerical example is provided in Section 5. Section 6 concludes the paper.
	
\section{Preliminaries and problem statement }

	Let $T \in (0, \infty)$ be fixed and $(\Omega, \mathbb{F}, \{\mathcal{F}_t\}_{0 \leq t \leq T}, \mathbb{P})$ be a complete filtered probability space. $B(\cdot)$ is a standard $\mathbb{R}$-valued Brownian motion. The filtration $\{\mathcal{F}_t\}_{0 \leq t \leq T}$ is generated by the Brownian motion $B(\cdot)$ and augmented by all the $\mathbb{P}$-null sets in $\mathbb{F}$, i.e. $\mathcal{F}_t := \sigma\{B(s): 0\leq s \leq t\} \vee \mathcal{N}(\mathbb{P})$. $f_x, f_{xx}$ denote the first and second partial derivatives with respect to $x$ for a differential function $f$. $L^2_{\mathcal{F}_t}(\Omega)$ denotes the set of $\mathbb{R}$-valued, $\mathcal{F}_t$-measurable, square integrable random variables; $L^2_{\mathbb{F}}(0,T)$ denotes the set of $\mathbb{R}$-valued, $\mathcal{F}_t$-adapted, square integrable processes on $[0,T]$; $L^{\infty}_{\mathbb{F}}(0,T)$ denotes the set of $\mathbb{R}$-valued, $\mathcal{F}_t$-adapted, bounded processes on $[0,T]$.
	
\subsection{Preliminaries}
	
	To formulate the problem and handle the advanced variable $Y$, we use the forward integral and Donsker delta function. Here, we introduce the definitions and some results. For more details, see, e.g., \cite{Nunno-Oksendal-Proske-2009}, \cite{Draouil-Oksendal-2015}, \cite{Draouil-Oksendal-2016} and the references therein.
	
	\begin{mydef}
		Let $Y: \Omega \rightarrow \mathbb{R}$ be a random variable which belongs to \textit{the Hida distribution space} $(\mathcal{S})^*$. Then a continuous functional $\delta_Y(\cdot):\mathbb{R} \rightarrow (\mathcal{S}^*)$ is called a \textit{Donsker delta function} of $Y$ if it has the property that
		\begin{equation}
			\int_{\mathbb{R}} g(y)\delta_Y(y)dy = g(Y), \quad a.s.,
		\end{equation}
		for all (measurable) $g:\mathbb{R} \rightarrow \mathbb{R}$ such that the integral converges.
	\end{mydef}
	
	\begin{Remark}
		(1) We can derive a connection between the conditional distribution of $Y$ and its Donsker delta function, as stated in Proposition 2.2 of \cite{Draouil-Oksendal-2015}.
		(2) When $Y$ is a Gaussian random variable, its Donsker delta function admits an explicit expression, which can be derived via the Hermite transform; see \cite{Aase-Oksendal-Uboe-2001}.
	\end{Remark}
		
	\begin{mydef}
		A stochastic process $\varphi(t), t\in [0, T]$, is said to be forward integrable with respect to Brownian motion $B(t)$ on $[0, T]$, if the following holds:
		
		(i) Weak sense: There exists a process $I(t), t\in [0, T]$, such that 
		\begin{equation}
			\sup\limits_{t \in [0, T]} \bigg|\int_{0}^{t} \varphi(s) \frac{B((s+\epsilon)\wedge t)- B(s)}{\epsilon} ds - I(t) \bigg| \rightarrow 0, \quad \epsilon \rightarrow 0^+,
		\end{equation} 
		in probability. 
		
		(ii) Strong sense: The limit 
		\begin{equation}
			\lim\limits_{\varepsilon \rightarrow 0^+} \int_0^T \varphi(t)\dfrac{B((t+\varepsilon)\wedge t)}{\varepsilon}dt
		\end{equation}
		exists in $L^2(\Omega)$.
		
		In both cases, we denote the limit by 
		\begin{equation}
			I(t) := \int_{0}^{t} \varphi(s) d^-B(s), \quad t \in [0, T],
		\end{equation}
		and call it the forward integral of $\varphi$ with respect to $B$.
	\end{mydef}
	
	The forward integral is an extension of the integral with respect to a semimartingale.
	\begin{mylem}
		 Let $\mathbb{G}:=\left\{\mathcal{G}_t, t \in[0, T]\right\}(T>0)$ be a given filtration. Suppose that
		 
		1. $B$ is a semimartingale with respect to the filtration $\mathbb{G}$.
		
		2. $\varphi$ is $\mathbb{G}$-predictable and the integral
		\begin{equation*}
			\int_0^T \varphi(t) dB(t)
		\end{equation*}
		with respect to $B$ exists. Then $\varphi$ is forward integrable and
		\begin{equation}
			\int_0^T \varphi(t) d^{-} B(t)=\int_0^T \varphi(t) dB(t) .
		\end{equation}
	\end{mylem}
	
	As a consequence of the above, we get the following useful result.
	\begin{mylem}\label{property}
		Let $\varphi(t, y)$ be an $\mathbb{F}$-adapted process for each $y \in \mathbb{R}$ such that
		\begin{equation*}
			\int_0^T \varphi(t, y) dB(t)
		\end{equation*}
		exists for each $y \in \mathbb{R}$. Let $Y$ be a random variable. Then $\varphi(t, Y)$ is forward integrable and
		\begin{equation}
			\int_0^T \varphi(t, Y)d^{-} B(t)=\int_0^T \varphi(t, y) dB(t)\Big|_{y=Y}.
		\end{equation}
	\end{mylem}
	 
	In the sequel, we give the It\^o's formula for forward integrals, which was first proved in \cite{Russo-Vallois-2000}.
	\begin{mydef}
		A forward process is a stochastic process of the form
		\begin{align}
			X(t)=x+\int_0^t b(s) ds+\int_0^t \sigma(s) d^{-} B(s), \quad  t \in[0, T], \label{forward_pro}
		\end{align}
		($x$ is a constant), where $\int_0^T|b(s)|ds<\infty$, a.s., and $\sigma$ is a forward integrable stochastic process. A shorthand notation for \eqref{forward_pro} is that
		\begin{align}
			d^{-} X(t)=b(t) dt+\sigma(t) d^{-} B(t). 
		\end{align}
	\end{mydef}
	
	\begin{mythm}\normalfont\textbf{(One-dimensional It\^o's formula for forward integrals)}
		Let $X(t)$ be a forward process with the form of \eqref{forward_pro}. Let $f \in \mathbf{C}^{1,2}([0, T] \times \mathbb{R})$ and define
		\begin{align}
			Y(t)=f(t, X(t)), \quad t \in[0, T].
		\end{align}
		Then $Y(t), t \in[0, T]$, is a forward process and
		\begin{align}
			d^{-} Y(t)=\frac{\partial f}{\partial t}(t, X(t)) dt + \frac{\partial f}{\partial x}(t, X(t)) d^{-} X(t)+\frac{1}{2} \frac{\partial^2 f}{\partial x^2}(t, X(t)) \sigma^2(t) dt .
		\end{align}
	\end{mythm}

\subsection{Problem statement}

	In this work, we consider a stochastic optimal control problem with inside information, motivated by insider trading phenomena in financial markets. To specify the problem, we consider the state dynamics via the following SDE: 
	\begin{align}
		\begin{cases}
			d^{-}X(t) = b(t, X(t), u(t), Y)dt + \sigma(t, X(t), u(t), Y)d^{-}B(t),\\
			X(0) = x_0,
		\end{cases}\label{state}
	\end{align} 
	where $Y \in L^2_{\mathcal{F}_{T_0}}(\Omega, \mathbb{R})$, $T_0 > T$, represents the inside information, i.e., future information available to the controller in advance. The control process $u(\cdot) = u(\cdot, Y)$ depends on both $Y$ and $\mathcal{F}_t$ and takes values in a nonempty convex set $U \subseteq \mathbb{R}$. For each $y \in \mathbb{R}$, $u(t, y)$ is $\mathcal{F}_t$-adapted. As discussed in Section \ref{sec1}, all stochastic integrals here are interpreted as forward integrals.
		
	In this framework, the controller has a performance functional given by
		\begin{align}
			J(u(\cdot, Y)) = \mathbb{E}\bigg[\int_0^T f(t, X(t), u(t, Y), Y)dt + g(X(T), Y)\bigg].\label{cost}
		\end{align}
		
	Here, $x_0 \in \mathbb{R}$, $b(t, x, u, y): \Omega \times [0, T] \times \mathbb{R} \times U \times \mathbb{R} \rightarrow \mathbb{R}$, $\sigma(t, x, u, y): \Omega \times[0, T] \times \mathbb{R} \times U \times \mathbb{R} \rightarrow \mathbb{R}$, $f(t, x, u, y): \Omega \times [0, T] \times \mathbb{R} \times U \times \mathbb{R} \rightarrow \mathbb{R}$ are given $\mathcal{F}_t$-adapted, $g(x, y): \Omega \times \mathbb{R} \times \mathbb{R} \rightarrow \mathbb{R}$ is $\mathcal{F}_T$-measurable, for each $(x, u, y)$. 
	
	We introduce the following assumption.
	
	\textbf{(H1)} $b, \sigma, f, g$ are twice continuously differentiable with respect to $x$. They and their derivatives $b_x, b_{x x}, \sigma_x, \sigma_{x x}, f_{ x}, f_{x x}, g_{ x}, g_{ x x}$ are continuous in $\left(t, x, u\right)$. Moreover, $b_x$, $b_{x x}$, $\sigma_x$, $\sigma_{x x}$,$ f_{x x}$, $g_{x x}$ are bounded by some constant $L>0$, and $b$, $\sigma$, $f_{x}$, $g_{x}$ are bounded by $L\left(1+|x|+\left|u\right|\right)$.
		
	Denote $\mathbb{H} := \{\mathcal{H}_t\}_{t \geq 0}$, where $\mathcal{H}_t := \mathcal{F}_t \vee \sigma(Y)$. To ensure its right-continuity, we set $\mathcal{H}_t = \mathcal{H}_{t^+} = \bigcap \limits_{s>t}\mathcal{H}_s$. Then we define the admissible control set as follows:
		\begin{align*}
			\mathcal{U} := \Big\{u \bigm| &u(t, Y):\Omega \times [0, T] \times \mathbb{R} \rightarrow U  \text{ is } \mathcal{H}_t \text{-progressive measurable, }\\
			&\sup\limits_{0 \leq t \leq T} \mathbb{E}|u(t, Y)|^i < \infty, i = 1, 2, \cdots\Big\}.
		\end{align*}
		
	\noindent\textbf{Optimal control problem under inside information}\quad {\it To find an admissible control $u^*(\cdot) \in \mathcal{U}$ satisfying \eqref{state} for the controller who has inside information $Y$, such that
		\begin{align}
			J(u^*(\cdot, Y)) = \max\limits_{u \in \mathcal{U}} J(u(\cdot, Y)).
		\end{align}}
		
	We also assume that $Y$ has a Donsker delta function $\delta_Y(y) \in (\mathcal{S})^*$. Then using $\delta_Y(y)$, we can rewrite the $\mathbb{H}$-adapted process $X(t)$ as follows:
		\begin{equation}
			X(t) = X(t, Y) = X(t, y)\big|_{y=Y} = \int_{\mathbb{R}} X(t, y)\delta_Y(y)dy,
		\end{equation}
	where $X(t, y)$ is $\mathbb{F}$-adapted for any $y \in \mathbb{R}$. Utilizing the definition of $\delta_Y(y)$ and Lemma \ref{property}, we derive
		\begin{align*}
			X(t) &= x_0 + \int_0^t b\big(s, X(s), u(s), Y)\big)ds + \int_0^t \sigma\big(s, X(s), u(s), Y)\big)d^{-}B(s)\\
				&= x_0 + \int_0^t b\big(s, X(s, Y), u(s, Y), Y)\big)ds + \int_0^t \sigma\big(s, X(s, Y), u(s, Y), Y)\big)d^{-}B(s)\\
				&= x_0 + \bigg(\int_0^t b\big(s, X(s, y), u(s, y), y)\big)ds\bigg)\bigg|_{y=Y} + \bigg(\int_0^t \sigma\big(s, X(s, y), u(s, y), y)\big)dB(s)\bigg)\bigg|_{y=Y}\\
				&= x_0 +  \int_{\mathbb{R}}\bigg[\int_0^t b\big(s, X(s, y), u(s, y), y)ds + \int_0^t \sigma\big(s, X(s, y), u(s, y), y)\big)dB(s)\bigg]\delta_Y(y)dy.
		\end{align*}
	Noting that the left-hand side of the equation equals $\int_{\mathbb{R}}X(t, y)\delta_Y(y)dy$, we introduce the following It\^o's type SDE:
		\begin{align}
			\begin{cases}
				dx(t, y) = b(t, x(t, y), u(t, y), y)dt + \sigma(t, x(t, y), u(t, y), y)dB(t),\\
				x(0, y) = x_0.
			\end{cases}\label{y_state}
		\end{align}
		
	The following proposition addresses the existence of solutions to the state equation \eqref{state}.	
		\begin{mypro}
			For given $u(\cdot, y)$ and $y\in\mathbb{R}$, if $x(t, y)$ satisfies \eqref{y_state}, then the process $\widetilde{X}(t) = x(t, Y), t\in [0, T]$ is a solution of \eqref{state}.
		\end{mypro}
	\begin{proof}
		Under assumption (H1), for each $y \in \mathbb{R}$, equation \eqref{y_state} possesses a unique solution $x(t, y) \in L^2_{\mathcal{F}}(0, T; \mathbb{R})$. We then rewrite the system in integral form as follows:
		\begin{align*}
			x(t, y) = x_0 +  \int_0^t b(s, x(s, y), u(s, y), y)ds + \int_0^t \sigma(s, x(s, y), u(s, y), y)\big)dB(s),
		\end{align*}
		where $\sigma(\cdot, x(\cdot, y), u(\cdot, y), y)$ is $\mathbb{F}$-adapted. By Lemma \ref{property}, replacing $y$ by $Y$ yields
		\begin{align*}
			x(t, Y) & =x_0+\int_0^t b(s,x(s,Y),u(s,Y),Y)dt +\bigg(\int_0^t\sigma(s,x(s,y),u(s,y),y)dB(s)\bigg)\bigg|_{y=Y} \\
			& =x_0+\int_0^t b(s, x(s, Y), u(s, Y), Y) d t+\int_0^t \sigma(s, x(s, Y), u(s, Y), Y) d^- B(s).
		\end{align*}
		Hence, $\widetilde{X}(t) = x(t, Y)$ satisfies \eqref{state}.
	\end{proof}
	
	\begin{Remark}
		We usually prove the uniqueness of solutions for It\^o's type SDEs by applying It\^o's formula to $|X_1-X_2|^2$ and conducting moment estimates. Unfortunately, evaluating the expectation of a forward integral introduces extra Malliavin derivative terms associated with the state process $X_i$, $i=1,2$, which complicates the analysis substantially.  Russo and Vallois \cite{Russo-Vallois-2000} prove the uniqueness of solutions to forward SDEs via a flow transformation technique where the integrand of the diffusion term only depends on $x$. For ease of reading, we present the uniqueness proof in the Appendix.
	\end{Remark}
 
	Similarly, we can derive the $y$-parameterized performance functional:
	\begin{align*}
		J(u(\cdot)) &= \mathbb{E}\bigg[ \int_{0}^{T} f\big(t, X(t), u(t), Y\big) dt + g\big(X(T), Y\big)\bigg]\\
		&=\mathbb{E}\bigg[ \bigg(\int_{0}^{T} f(t, x(t, y), u(t,y), y) dt\bigg)\bigg|_{y=Y} + g\big(x(T,y), y\big)\big|_{y=Y}\bigg]\\
		&=\mathbb{E}\bigg\{\int_{\mathbb{R}} \bigg[\int_{0}^{T} f\big(t, x(t, y), u(t,y), y\big)\delta_Y(y)dt +  g\big(x(T,y), y\big)\delta_Y(y)\bigg]dy\bigg\} \\
		&=\int_{\mathbb{R}}\mathbb{E} \bigg[\int_{0}^{T} f\big(t, x(t, y), u(t,y), y\big)\mathbb{E}[\delta_Y(y)|\mathcal{F}_t]dt +  g\big(x(T,y), y\big)\mathbb{E}[\delta_Y(y)|\mathcal{F}_T]\bigg]dy.
	\end{align*}
	Denote the new performance functional and admissible control set as follows:
	\begin{align}
		&\widehat{J}(u(\cdot, y)) = \mathbb{E} \bigg[\int_{0}^{T} f\big(t, x(t, y), u(t,y), y\big)\mathbb{E}[\delta_Y(y)|\mathcal{F}_t]dt +  g\big(x(T,y), y\big)\mathbb{E}[\delta_Y(y)|\mathcal{F}_T]\bigg],\label{y_performance}\\
		&\mathcal{A} := \Bigl\{u(t, y)\bigm|u(t):[0, T]\times\mathbb{R}\times\Omega \rightarrow \mathbb{R} \text{ is } \mathcal{F}_t \text{-progressive measurable}, \nonumber\\
		&\qquad\qquad\qquad\sup\limits_{0 \leq t \leq T} \mathbb{E}|u(t, Y)|^i < \infty, i = 1, 2, ...\Bigr\}.
	\end{align}
	Then the original problem is transformed into the following one.

	\noindent\textbf{$y$-parameterized Problem}\quad {\it For each $y \in \mathbb{R}$, find $u^*(\cdot, y) \in \mathcal{A}$ satisfying \eqref{y_state} such that \eqref{y_performance} is maximized, i.e.,
	\begin{align*}
		\widehat{J}(u^*(\cdot, y)) = \sup\limits_{u \in \mathcal{A}} \widehat{J}(u(\cdot, y)).
	\end{align*}}
	 
\section{A Pontryagin’s maximum principle}

	Let $\varphi(\cdot) \in L^2(0, T_0)$ be deterministic and assume that $\|\varphi\|^2_{[t, T]} := \int_t^T \varphi(s)^2 ds > 0$. Let $Y$ be a Gaussian random variable of the form 
	\begin{align}
		Y = Y(T_0),\quad \text{where } Y(t) = \int_0^t \varphi(s)dB(s), \quad \text{for } t \in [0, T_0].\label{Y}
	\end{align}
	\cite{Aase-Oksendal-Uboe-2001} showed that its Donsker delta function is given by
	\begin{align}
		\delta_Y(y) = (2\pi \Vert \varphi \Vert^2_{[0,T_0]})^{-\frac{1}{2}}\exp^{\diamond }\left[-\frac{(Y-y)^{\diamond 2}}{2\Vert \varphi \Vert^2_{[0,T_0]}}\right],\label{Y_donsker}
	\end{align}
	where the symbol $\diamond$ denotes the Wick product. The corresponding conditional expectations read
	\begin{align}
		\mathbb{E}\big[\delta_Y(y)|\mathcal{F}_t\big]&= \left(2\pi \Vert \varphi \Vert^2_{[t,T_0]}\right)^{-\frac{1}{2}}
		\exp\left[-\frac{(Y(t)-y)^2}{2\Vert \varphi \Vert^2_{[t,T_0]}}\right],\label{EY_donsker}\\
		\mathbb{E}\big[D_t\delta_Y(y)|\mathcal{F}_t\big]&= -\left(2\pi \Vert \varphi \Vert^2_{[t,T_0]}\right)^{-\frac{1}{2}}
		\exp\left[-\frac{(Y(t)-y)^2}{2\Vert \varphi \Vert^2_{[t,T_0]}}\right]
		\frac{Y(t)-y}{\Vert \varphi \Vert^2_{[t,T_0]}}\varphi(t),\label{EDY_donsker}
	\end{align}
	where $\mathbb{E}\big[\delta_Y(y)|\mathcal{F}_t\big]$ is positive and uniformly bounded on $[0,T]$, $\mathbb{E}\big[D_t\delta_Y(y)|\mathcal{F}_t\big]\in L^2_{\mathcal{F}_t}(\Omega, \mathbb{R})$. For details, see \cite{Aase-Oksendal-Uboe-2001}, Proposition 3.2, Lemma 3.7 and 3.8, \cite{Draouil-Oksendal-2015}, Section 2.
	
	It is not difficult to obtain the following maximum principle for the $y$-parameterized Problem from the classical result of Peng \cite{Peng-1990}. Thus, we omit the detailed deduction and only state the main result. Define the Hamiltonian function $H: \Omega \times [0, T] \times \mathbb{R} \times U \times \mathbb{R} \times \mathbb{R} $ by
	\begin{align*}
		H(t, x, u, p, q) = f(t, x, u, y)\mathbb{E}\big[\delta_Y(y)|\mathcal{F}_t\big] + b(t, x, u, y)p + \sigma(t, x, u, y)q.
	\end{align*}
	For a given admissible control $u \in \mathcal{A}$, we let $x$ be the corresponding state trajectory. The adjoint $\mathcal{F}_t$-adapted process pair $\big(p(\cdot, y), q(\cdot, y)\big)$ satisfies
	\begin{align}
		\begin{cases}
			-dp(t, y) =\bigg[b_x(t, x(t, y) u(t, y), y)^\mathrm{T}p(t, y) + \sigma_x(t, x(t, y) u(t, y), y)^\mathrm{T}q(t, y) \\
			\qquad \qquad \qquad+ f_x(t, x(t,y), u(t,y), y)\mathbb{E}\big[\delta_Y(y)|\mathcal{F}_t\big]\bigg]dt -q(t, y)dB(t),\\
			p(T, y) = g_x(x(T, y))\mathbb{E}\big[\delta_Y(y)|\mathcal{F}_T\big].
		\end{cases}\label{adjoint}
	\end{align}
	\begin{mypro}
		Suppose that \textbf{(H1)} holds, for any $y \in \mathbb{R}$, let $u^*(\cdot, y)$ be the optimal control for the \textbf{Problem}, $x^*(\cdot, y)$ be the corresponding optimal state. Then 
		\begin{align}
			H_u(t, x^*(t,y), u^*(t,y), p(t,y), q(t,y))\big(u-u^*(t,y) \big) \leq 0, \quad a.e.\ t \in [0, T], 
		\end{align}
		holds for any $u \in U$, a.e., a.s., where $(p(\cdot), q(\cdot))$ is the solution of \eqref{adjoint}.
	\end{mypro}
	
	We then derive Arrow’s sufficient optimality condition, referring to the work of Wang and Yu \cite{WangGC-YuZY-2010}, who proved Arrow-type sufficient optimality condition for a nonzero sum BSDE stochastic differential game. Shi and Wang \cite{Shi-Wang-2016} extended this research by deriving the corresponding sufficient optimality condition for BSDE games with delayed generators.
	\begin{mypro}
		Let \textbf{(H1)} hold and $u^*(\cdot,y) \in \mathcal{A}$ be given. Suppose that $x^*(\cdot,y)$, $\big(p(\cdot, y), q(\cdot, y)\big)$ are the corresponding solutions to \eqref{y_state} and \eqref{adjoint}. Suppose that
		\begin{align}
			H(t, x^*(t,y), u^*(t,y), p(t,y), q(t,y)) = \max\limits_{u \in U} H(t, x^*(t,y), u, p(t,y), q(t,y)) \label{sufficient1}
		\end{align}
		holds for all $t \in [0, T]$, and  
		\begin{align}
			\widehat{H}(t, x) = \max\limits_{u \in U} H(t, x, u, p(t,y), q(t,y)) \label{sufficient2}
		\end{align}
		exists and is concave in $x$, $g$ is concave in $x$. Then $u^*(\cdot, y)$ is optimal.
	\end{mypro}
	\begin{proof}
		Let $x(\cdot,y )$ be the corresponding state process to $u(\cdot) \in \mathcal{U}$, we need to prove that
		\begin{align*}
			\widehat{J}\big(u^*(\cdot, y)\big) - \widehat{J}\big(u(\cdot,y)\big) = \mathrm{I} + \mathrm{II} \geq 0,
		\end{align*}
		where
		\begin{align*}
			\mathrm{I}
			&:= \mathbb{E}\int_0^T \Big[f\big(t, x^*(t,y), u^*(t,y), y\big) - f\big(t, x(t,y), u(t,y), y\big)\Big]\mathbb{E}\big[\delta_Y(y)|\mathcal{F}_t\big]dt,\\
			\mathrm{II}
			&:= \mathbb{E}\Big[ \big( g\big(x^*(T)\big) - g\big(x(T)\big)\big)\mathbb{E}\big[\delta_Y(y)|\mathcal{F}_T\big]\Big].
		\end{align*}
		Denote by $\phi(t) = \phi(t, x^*(t,y), u^*(t,y),y), \phi = b, \sigma, f, b_x, \sigma_x, f_x$ and $\alpha(t) = x^*(t,y) - x(t,y)$, then $\alpha(t)$ is the unique solution to 
		\begin{align}
			\begin{cases}
				d\alpha(t) = \big[b_x(t)\alpha(t) + \beta(t)\big]dt + \big[\sigma_x(t)\alpha(t)  + \rho(t)\big]dB(t),\\
				\alpha(0) = 0,
			\end{cases}
		\end{align}
		with 
		\begin{align*}
			\beta(t) &:= -b_x(t)\alpha(t) + b(t) - b(t, x(t,y), u(t,y),y),\\
			\rho(t) &:= -\sigma_x(t)\alpha(t) + \sigma(t) -\sigma(t, x(t,y), u(t,y),y).
		\end{align*}
		Applying It\^o's formula to $ p(\cdot,y)\big(x^*(\cdot,y) - x(\cdot,y)\big)$, we can derive
		\begin{align*}
			&\mathbb{E}\Big[ g_x(x^*(T,y))\mathbb{E}\big[\delta_Y(y)|\mathcal{F}_T\big]\big( x^*(T,y)-x(T,y) \big)\Big]\\
			&=\mathbb{E}\int_0^T\Big[
				- f_x(t)\mathbb{E}\big[\delta_Y(y)|\mathcal{F}_t\big] \alpha(t) -  p(t,y)b_x(t)\alpha(t) -  q(t,y)\sigma_x(t)\alpha(t) \nonumber\\
				&\qquad+ p(t,y)\big(b(t)-b(t, x(t,y),u(t,y),y) \big) +  q(t,y)\big( \sigma(t)-\sigma(t, x(t,y),u(t,y),y) \big)\Big]dt\\
			&=\mathbb{E}\int_0^T\Big[
				-  H_x(t, x^*(t,y),u^*(t,y),p(t,y),q(t,y))\alpha(t)\nonumber\\
				&\qquad+ p(t,y)\big(b(t)-b(t, x(t,y),u(t,y),y)\big) +  q(t,y) \big(\sigma(t)-\sigma(t, x(t,y),u(t,y),y) \big)\Big]dt.
		\end{align*}
		Since $g$ is concave in $x$, we have
		\begin{align*}
			\mathrm{II} \geq \mathbb{E}\Big[ g_x(x^*(T,y))\mathbb{E}\big[\delta_Y(y)|\mathcal{F}_T\big] \big( x^*(T,y)-x(T,y) \big)\Big].
		\end{align*}
		Then
		\begin{align*}
			&\widehat{J}\big(u^*(\cdot, y)\big) - \widehat{J}\big(u(\cdot,y)\big)\\
			&\quad \geq \mathbb{E} \int_0^T \bigg[ H(t,x^*(t,y),u^*(t,y),p(t,y),q(t,y)) - H(t,x(t,y),u(t,y),p(t,y),q(t,y)) \\
			&\qquad -  H_x(t, x^*(t,y),u^*(t,y),p(t,y),q(t,y))\alpha(t) \bigg]dt.
		\end{align*}
		Under conditons \eqref{sufficient1}) and \eqref{sufficient2}, for each $(t, x, y)$
		\begin{align}
			&H(t, x^*(t,y), u^*(t,y),p(t,y),q(t,y)) - H(t, x, u(t,y),p(t,y),q(t,y)) \nonumber \\
			&\qquad  \geq \widehat{H}(t, x^*(t,y)) - \widehat{H}(t, x).
		\end{align}
		Fix $(t, y)$. Since $\widehat{H}$ is concave in $x$, there exists a supergradient $\theta(t) \in \mathbb{R}$ for $\widehat{H}$ at $x^*(t,y)$, such that for all $x \in \mathbb{R}$
		\begin{align}
			\widehat{H}(t, x) - \widehat{H} (t, x^*(t,y)) - \theta(t)\big( x-x^*(t,y) \big) \leq 0. \label{hat_H}
		\end{align}
		Define 
		\begin{align*}
			&\gamma(t, x) \\
			&:= H(t, x, u^*(t,y), p(t,y), q(t,y)) - H(t, x^*(t,y), u^*(t,y),p(t,y),q(t,y))  - \theta(t)\big( x-x^*(t,y) \big)\\
			&\leq\widehat{H}(t, x) - \widehat{H}(t, x^*(t,y)) - \theta(t)\big( x-x^*(t,y) \big),
		\end{align*}
		with \eqref{hat_H}, $\gamma(t, x) \leq 0$ and $\gamma(t, x^*(t,y)) = 0$. Thus, $x^*(t, y)$ is the maximum point of $\gamma$. Because $\gamma$ is differentiable in $x$, we have $\gamma_x(t, x^*(t,y)) =0$, which means
		\begin{align*}
			H_x(t, x^*(t,y), u^*(t,y), p(t,y), q(t,y)) = \theta(t).
		\end{align*}
		Therefore
		\begin{align*}
			&\widehat{J}\big(u^*(\cdot, y)\big) - \widehat{J}\big(u(\cdot,y)\big)\\
			&\quad \geq \mathbb{E} \int_0^T \Big[\widehat{H}(t, x^*(t,y)) - \widehat{H}(t, x(t,y)) - \theta(t)\big( x^*(t,y) - x(t,y) \big)\Big] \geq 0.
		\end{align*}
        The proof is complete.
	\end{proof}

\section{LQ case under inside information}

	Let us consider the following linear system: 
	\begin{align}
		\begin{cases}
			d^{-}X(t) = \big[A(t)X(t) + B(t)u(t, Y) \big]dt + \big[C(t)X(t) + D(t)u(t, Y) \big]d^{-}B(t),\\
			X(0) = x_0,  \quad t \in [0, T].
		\end{cases}\label{lqstate}
	\end{align}
    We wish to choose a $u(\cdot,Y) \in \mathcal{U}$ to maximize the quadratic performance functional given by
	\begin{align}
		J(u(\cdot)) = -\frac{1}{2} \mathbb{E} \bigg\{\int_0^T \Big[Q(t)(X(t))^2 + R(t)(u(t, Y))^2\Big]dt + G(X(T, Y))^2\bigg\}.\label{lqcost}
	\end{align}
	The following assumption is needed.	
	
	\textbf{(H2)} $A(\cdot), B(\cdot), C(\cdot), D(\cdot)$ are deterministic, continuous functions on $[0, T]$; $Q(\cdot), R(\cdot)$ are deterministic, and continuous with respect to $t \in [0,T]$;  $Q(\cdot), G \geq 0$ and $R(\cdot) \gg 0$. 
	
	The problem can be written in detail as follows.
	
	\noindent\textbf{Problem (LQ-II)} Find $u^*(t, Y) \in \mathcal{U}$ satisfying \eqref{lqstate} such that \eqref{lqcost} achieves the maximum, namely
	\begin{align*}
		J(u^*(\cdot, Y)) = \sup\limits_{u \in \mathcal{U}} J(u(\cdot, Y)).
	\end{align*}

	Firstly, using the Donsker delta function of $Y$, we construct the $y$-parameterized state equation and the associated performance functional for the aforementioned LQ problem. Let $x(t) := x(t, y), u(t) :=u(t, y)$, then
	\begin{align}
		\begin{cases}
			dx(t) = \big[A(t)x(t) + B(t)u(t) \big]dt + \big[C(t)x(t) + D(t)u(t)\big]dB(t),\\
			x(0) = x_0,
		\end{cases} \label{y-lqstate}
	\end{align}
	\begin{align}
		\widehat{J}(u(\cdot)) = -\frac{1}{2} \mathbb{E} \bigg\{
			\int_0^T \Big[ Q(t)(x(t))^2 + R(t)(u(t))^2\Big]\mathbb{E}[\delta_Y(y)|\mathcal{F}_t]dt + G(x(T))^2\mathbb{E}[\delta_Y(y)|\mathcal{F}_T]\bigg\}.\label{y-lqcost1}
	\end{align}
	The original system is converted into the classical SDE presented above. Under the preceding assumption, equation \eqref{y-lqstate} admits a unique solution for any given admissible control $u(\cdot)$ and fixed $y$. This transforms our original problem into an LQ optimal control problem with random coefficients, where the randomness stems from the conditional expectations $\mathbb{E}[\delta_Y(y)|\mathcal{F}_t]$ and $\mathbb{E}[\delta_Y(y)|\mathcal{F}_T]$. Denote $\widehat{Q} = Q\mathbb{E}[\delta_Y(y)|\mathcal{F}_t]$, $\widehat{R} = R\mathbb{E}[\delta_Y(y)|\mathcal{F}_t]$, $\widehat{G} = G\mathbb{E}[\delta_Y(y)|\mathcal{F}_T]$, then $\widehat{Q}, \widehat{R}, \widehat{G}$ are bounded random quantities, and $\widehat{Q}, \widehat{G} \geq 0$, $\widehat{R} > 0$. Then we rewrite \eqref{y-lqcost1} as
	\begin{align}
		\widehat{J}(u(\cdot)) = -\frac{1}{2} \mathbb{E} \bigg\{
			\int_0^T \Big[\widehat{Q}(t)(x(t))^2 + \widehat{R}(t)(u(t))^2\Big]dt + \widehat{G}(x(T))^2\bigg\}.\label{y-lqcost2}
	\end{align}

	We can write the Hamiltonian function
	\begin{align}
		H(t, x, u, p, q, y) =& \big[A(t)x(t) + B(t)u(t) \big]p(t) + [C(t)x(t) + D(t)u(t) ]q(t) \\ \nonumber
		&- \frac{1}{2} \Big[\widehat{Q}(t)(x(t))^2 + \widehat{R}(t)(u(t))^2\Big].
	\end{align}
	Therefore, by Proposition 3.1, if $u^*(\cdot) = u^*(\cdot,y)$ is an optimal control, then we have
	\begin{align}
		B(t)p(t) + D(t)q(t) - \widehat{R}(t)u^*(t) = 0.\label{y_u*1}
	\end{align}
	Here $(p(\cdot), q(\cdot))$ is the $\mathcal{F}_t$-adapted solution of the following adjoint BSDE:
	\begin{align}
		\begin{cases}
			-dp(t) = \big[A(t)p(t) + C(t)q(t) - \widehat{Q}(t)x^*(t)\big]dt - q(t)dB(t),\\
			p(T) = - \widehat{G}x^*(T).
		\end{cases}\label{y_p}
	\end{align}
	We then apply the four-step approach and introduce a SRE to derive the explicit closed-form expression for the optimal control $u^*(\cdot)$. In view of the terminal condition in \eqref{y_p}, we set
	\begin{align}
		- p(t) = P(t)x^*(t) , \quad t \in [0, T],\label{y_p1}
	\end{align}
	for some $\mathbb{F}$-adapted process pair $(P(\cdot), \Lambda(\cdot))$ satisfying the BSDE:
	\begin{align}
		dP(t) = -F(t)dt + \Lambda(t)dB(t), \quad P(T) = \widehat{G},\label{y-riccati}
	\end{align}
	where $F(\cdot)$ will be determined later. Applying It\^o's formula to \eqref{y_p1}, we obtain
	\begin{align}
		- dp(t) = \big[-Fx^* + PAx^* + PBu^* + \Lambda\big(Cx^* + Du^*\big)\big]dt + \big[P\big(Cx^* + Du^*) +\Lambda x^* \big]dB(t).\label{y_p2}
	\end{align}
	Comparing \eqref{y_p} with \eqref{y_p2}, we get
	\begin{align}
		0 &= -Fx^* + PAx^* + PBu^* + \Lambda\big(Cx^* + Du^*\big) - Ap - Cq + \widehat{Q}x^*, \label{y_p3}\\
		q &= -\big[P\big(Cx^* + Du^*) +\Lambda x^* \big], \quad a.s.\label{y_q}
	\end{align}
	Substituting \eqref{y_p1} and \eqref{y_q} into \eqref{y_u*1}, we obtain
	\begin{align*}
		BPx^* + D\big[P\big(Cx^* + Du^*) +\Lambda x^* \big] + \widehat{R}u^* 
		= \big( \widehat{R} + D^2P\big)u^* + \big[BP + D\big(PC+\Lambda\big)\big]x^* 
		= 0.
	\end{align*}
	Assume that $\widetilde{R}^{-1} \triangleq \big(\widehat{R}  + D^2P\big)^{-1}$ exists, then
	\begin{align}
		u^* =  - \widetilde{R}^{-1}\big[BP + D\big(PC+\Lambda\big)\big]x^*,\quad a.e.\ t \in [0, T],\ a.s.\label{y_u*2}
	\end{align}
	Substituting \eqref{y_p1}, \eqref{y_q}, and \eqref{y_u*2} into \eqref{y_p3}, we conclude that $P(\cdot)$ satisfies the following BSDE:
	\begin{align}
		\begin{cases}
			dP(t) = -\Big[2PA + 2C\Lambda + PC^2 +\widehat{Q} - \widetilde{R}^{-1}\big[BP+ D\big(PC+\Lambda\big)\big]^2\Big]dt + \Lambda dB(t),\\
			P(T) = \widehat{G},\quad t \in [0,T].
		\end{cases}\label{y_P}
	\end{align}
	\begin{Remark}
		Sun et al. \cite{Sun-Xiong-Yong-2021} proved that under the assumptions of bounded coefficients and uniform convexity of the cost functional with respect to the control, i.e., $J(t,0,u)\geq \delta\|u\|^2$, for some constant $\delta>0$, the aforementioned SRE admits a unique solution $(P, \Lambda) \in L_{\mathbb{F}}^{\infty}([0,T]; \mathbb{S}) \times L_{\mathbb{F}}^2([0,T];\mathbb{S})$, satisfying $\widetilde{R} \geq \lambda$, for some constant $\lambda>0$, a.e. $t\in[0,T]$.
	\end{Remark}
	
	Substituting \eqref{y_u*2} into \eqref{y-lqstate}, we derive the following closed-loop system:
	\begin{align}
		\begin{cases}
			dx^*(t) = \Big\{A(t) - B(t)\widetilde{R}^{-1}(t)\big[B(t)P(t) + D(t)\big(P(t)C(t)+\Lambda(t)\big)\big]\Big\}x^*(t)dt\\
			\qquad \qquad + \Big\{C(t) - D(t)\widetilde{R}^{-1}(t)\big[B(t)P(t) + D(t)\big(P(t)C(t)+\Lambda(t)\big)\big] \Big\}x^*(t)dB(t),\\
			x^*(0) = x_0.
		\end{cases}\label{y_x*}
	\end{align}
	We summarise the above results in the following statement.
	\begin{mythm}
		Let \textbf{(H2)} hold. Assume $\widehat{R}>\delta$, for some constant $\delta > 0$. Then \eqref{y_P} admits a unique adapted solution $(P(\cdot), \Lambda)$ such that $\widetilde{R}^{-1}$ is bounded, for each fixed $y\in\mathbb{R}$; the optimal feedback control $u^*(t, Y)$ for problem (LQ-II) reads
		\begin{align}
			u^*(t, Y) = \int_{\mathbb{R}} u^*(t, y)\delta_Y(y) dy,
		\end{align}
		where $u^*(\cdot, y)$ is given by \eqref{y_u*2} with $x^*(\cdot)$ being the solution of \eqref{y_x*}. Furthermore, the optimal performance is 
		\begin{align}
			\widehat{J}(u^*(\cdot, y)) = - \frac{1}{2}\mathbb{E}\big[P(0)x_0^2\big].\label{y_optimal_cost}
		\end{align}
	\end{mythm}
	\begin{proof}
		For any given admissible control $u(\cdot, y)$, let $x(\cdot)$ be the corresponding solution of \eqref{y-lqstate}. Applying It\^o's formula to $\frac{1}{2}P(\cdot)(x(\cdot))^2$, and integrating it from 0 to $T$, it follows that
		\begin{align}
			-\frac{1}{2}\mathbb{E}\Big[\widehat{G}(x(T))^2\Big] 
			&= - \frac{1}{2}\mathbb{E} \int_0^T \bigg[ -F(t)(x(t))^2 + 2P(t)x(t)\big(A(t)x(t)+B(t)u(t)\big) \nonumber\\
			&\quad+P(t)\big(C(t)x(t)+B(t)u(t)\big)^2 + 2\Lambda(t) x(t)\big(C(t)x(t)+B(t)u(t)\big) \bigg]dt \nonumber\\
			&\quad- \frac{1}{2}\mathbb{E}\Big[P(0)(x(0))^2\Big]. \label{Gx^2}
		\end{align}
		In view of \eqref{y-lqstate}, \eqref{y_P}, we substitute \eqref{Gx^2} into \eqref{y-lqcost2} and use the method of completion-of-squares, deriving that 
		\begin{align*}
			\widetilde{J}\Big(u(t,y)\Big) &= - \frac{1}{2} \mathbb{E}\int_0^T\bigg\{
				\left| 
				\widetilde{R}^{-\frac{1}{2}}(t) \Big[\widetilde{R}(t)u(t)+\Big(B(t)P(t) + D(t)\big(P(t)C(t)+\Lambda(t)\big)x(t)\Big)\Big] \right|^2 \bigg\}dt\\
				&\quad- \frac{1}{2}\mathbb{E}\big[P(0)x_0^2\big].
		\end{align*}
		It is obvious that for optimal $u^*(\cdot, y)$ given by \eqref{y_u*2}, the corresponding optimal performance functional is given by \eqref{y_optimal_cost}. Moreover, for any given admissible control $u(\cdot,y)$, we have $\widetilde{J}\big(u(\cdot,y)\big) \leq \widetilde{J}\big(u^*(\cdot, y)\big)$. The proof is complete.
	\end{proof}
	
	Recall that $Y = Y(T_0)$ satisfies equation \eqref{Y}. For any fixed $t \in [0,T]$, the conditional expectation $\mathbb{E}[\delta_Y(y)|\mathcal{F}_t]$ is a positive, bounded random variable,  where $\mathbb{E}\big[D_t\delta_Y(y)|\mathcal{F}_t\big]$ belongs to $L^2_{\mathcal{F}_T}(\Omega,;\mathbb{R})$. Based on the above theorem, we derive the optimal control $u^*$ for the LQ case. In addition, as shown in \O ksendal and Rose \cite{Oksendal-Rose-2017}, $B$ is an $\mathbb{H}$-semimartingale and admits a semimartingale decomposition.
	 
	\begin{Corollary}
		Let $Y$ satisfy \eqref{Y}. Then the $\mathbb{F}$-Brownian motion $B$ is a semimartingale with respect to the filtration $\mathbb{H}$ enlarged by the inside information, and it admits the following semimartingale decomposition:
		\begin{align}
			B(t) = \widetilde{B}(t) + \int_0^t \rho_s(Y) ds,
		\end{align}
		where $\widetilde{B}$ is an $\mathbb{H}$-Brownian motion, and $\rho_s$ (called the information drift) is of the form
		\begin{align}
			\rho_t(Y) = \frac{\mathbb{E}\big[D_t\delta_Y(y)|\mathcal{F}t\big]{y=Y}}{\mathbb{E}\big[\delta_Y(y)|\mathcal{F}t\big]{y=Y}} = \frac{Y(T_0) - Y(t)}{\left\|\varphi \right\|^2_{[t,T_0]}}\varphi(t).
		\end{align}
	\end{Corollary}
	
\section{A numerical example}

	In this section, we give one numerical example to demonstrate the impact of inside information $Y$ on the problem. Firstly, we give an algorithm to illustrate the process of solving LQ insider optimal control.
	
	\setlength{\algoheightrule}{1.2pt}
	\setlength{\algotitleheightrule}{1.2pt}
	
	\begin{algorithm}[H]
		\caption{Algorithm for the LQ optimal insider control.}
		\label{alg:insider-lq}
		\DontPrintSemicolon
		
		\textbf{Initialization:}
		Choose the time horizons $0<T<T_0$ and coefficients of the linear stochastic system \eqref{lqstate}-\eqref{lqcost}.
		
		Simulate inside information $Y$ using the Monte Carlo method.
		
		Solve $(P, \Lambda)$ by equation \eqref{y-riccati}.
	
		Calculate the optimal state trajectory $x^*$ by equation \eqref{y_x*}.
		
		Obtain the optimal feedback insider control $u^*$ via \eqref{y_u*2}.
	\end{algorithm}
	
	We set the time horizons as $T = 1$ and $T_0 = 1.8$, and $\varphi = 1$, i.e., $Y = B(T_0)$. The coefficients for the stochastic LQ system \eqref{lqstate}-\eqref{lqcost} are chosen as $x_0 = 1$, $A=0.05$, $B=0.8$, $C=0.7$, $D=0.5$, $Q=0.8$, $R=0.5$, and $G=1.5$. We next illustrate the impact of inside information on the system by comparing the informed controller, whose control depends on the inside information $Y$, with the uninformed controller, whose control is independent of \(Y\).

	Figure 1 compares the stochastic Riccati process $P(\cdot,Y)$ with inside information and $P^0(\cdot)$ which is the Riccati process for the uninformed controller. It can be observed that $P^0$ is a smooth deterministic curve. In contrast, the stochastic Riccati process $P$ depending on the inside information $Y$ constitutes a family of random trajectories, where the dark blue line represents a representative sample path.
	
	\begin{figure}[h!]	
		\centering
		\includegraphics[width=0.8\textwidth]{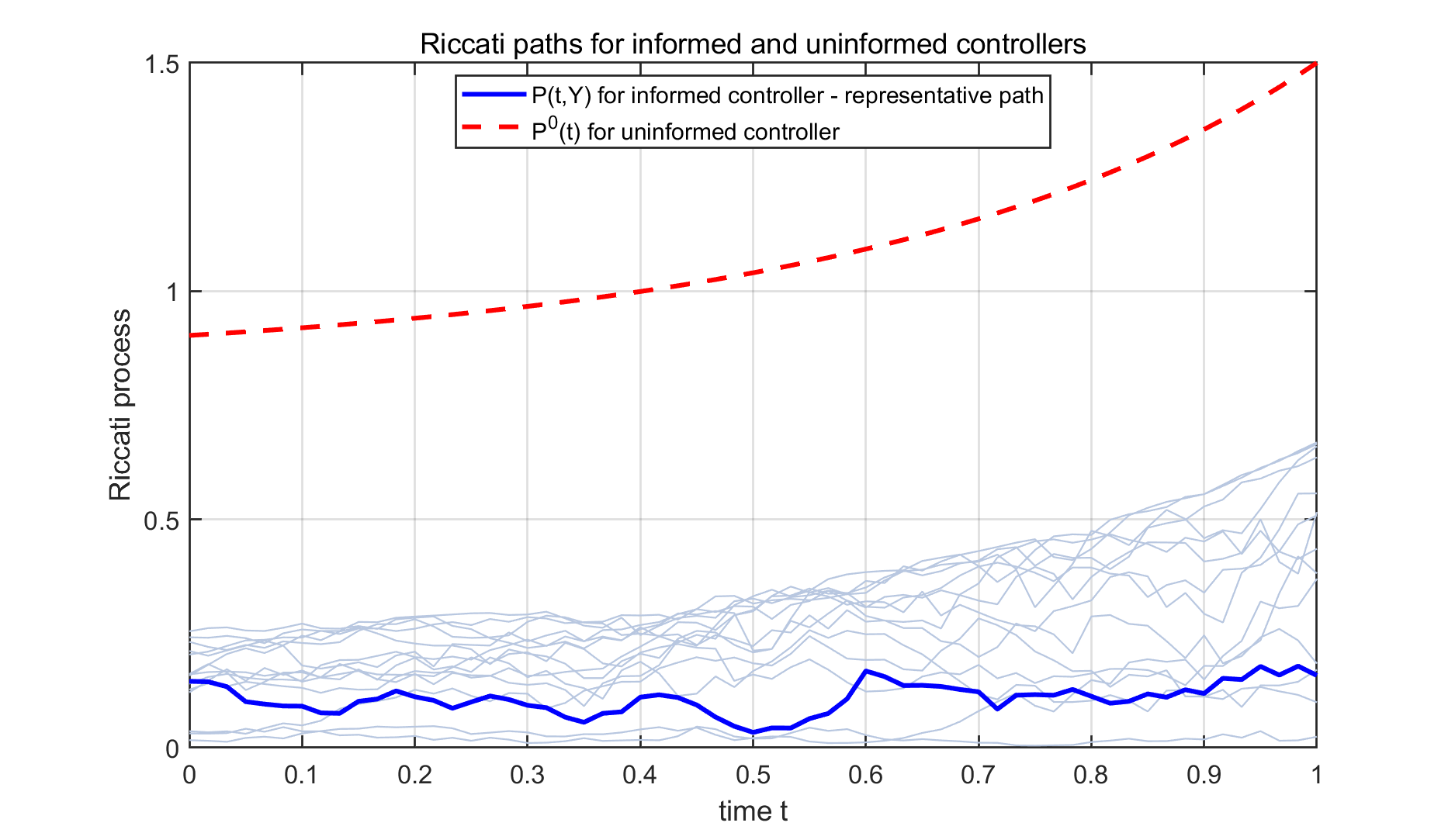} 
		\caption{The curves of $P(t,Y)$ and $P^0(t)$} 
		\label{} 
	\end{figure}
	
	Figure 2 compares the optimal control of the informed controller with access to the inside information $Y$ with that of the uninformed controller. The informed optimal control corresponds to a representative trajectory generated by the realization of inside information $Y(T_0)=1.342$. In the early stage, the absolute value of the informed control exceeds that of the uninformed control, which implies that inside information prompts the controller to adopt stronger control actions. As time approaches the terminal moment, the two curves gradually converge, indicating that the impact of information fades over time.
	
	\begin{figure}[h!]	
		\centering
		\includegraphics[width=0.8\textwidth]{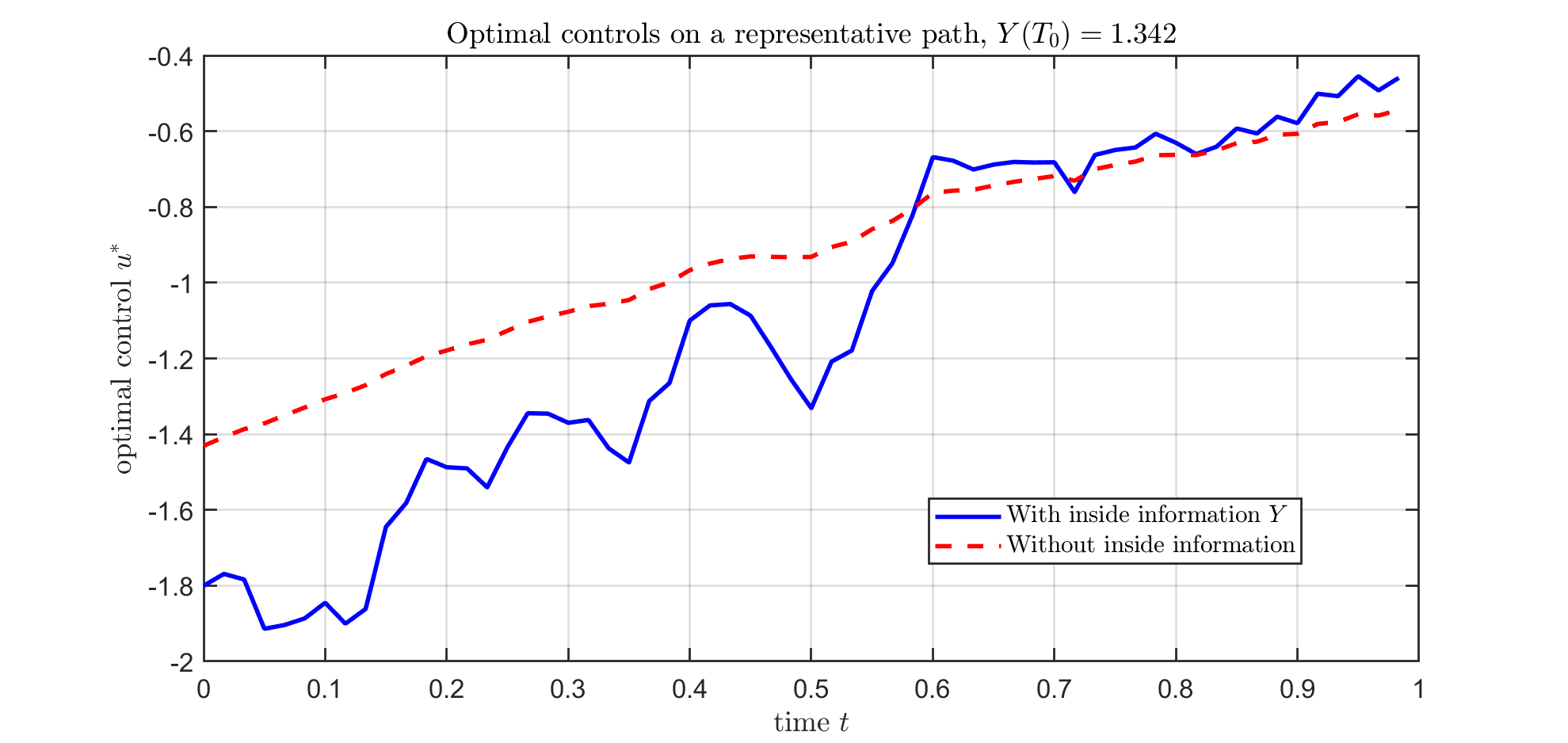} 
		\caption{The curves of $u(t,Y)$ and $u(t)$} 
		\label{} 
	\end{figure}
	
	In Figure 3, we compare the optimal performance (cost) of the informed controller and the uninformed controller. Both subfigures demonstrate that the informed controller who has private information yields a lower average cost $-J$ and a higher average performance $J$, indicating that inside information improves the performance in expectation sense.
	
	\begin{figure}[h!]	
		\centering
		\includegraphics[width=0.8\textwidth]{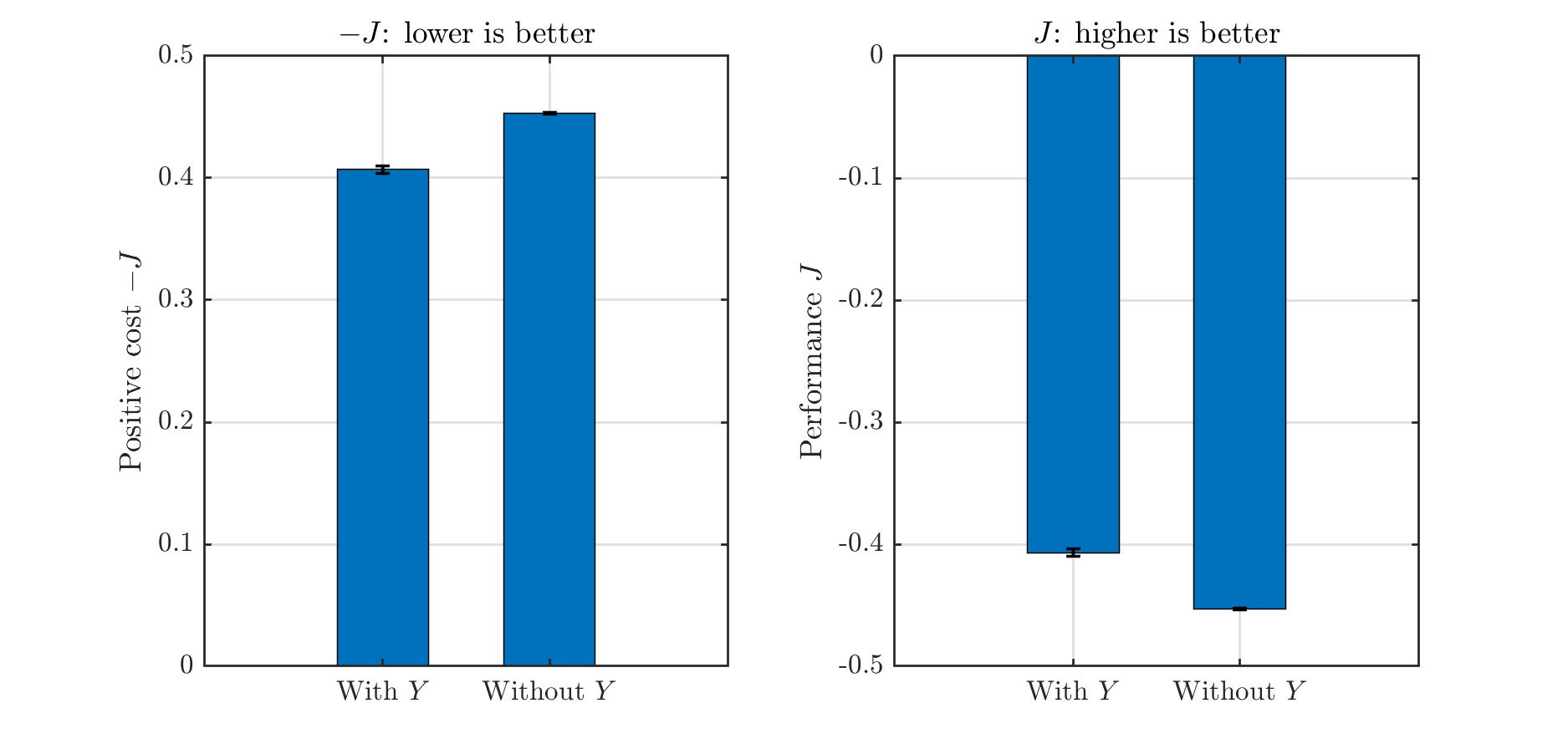} 
		\caption{The curves of $-J$ and $J$} 
		\label{} 
	\end{figure}
	
\section{Concluding remarks}

	In this paper, we have investigated a stochastic optimal control problem with inside information, where the controller has access to future information about the system. The inside information is represented by a static $\mathcal{F}_{T_0}$-measurable random variable $Y$. The control process $u(\cdot, Y)$ is thus adapted to the enlarged filtration generated by the underlying Brownian motion and $Y$, which renders standard stochastic integrals ill-defined. Instead, we adopt forward stochastic integrals to characterize the integral terms in the system. Using the Donsker delta function approach, we transform the original forward SDE into a $y$-parameterized It\^o's type SDE. The existence of solutions to the forward SDE is established via the solutions of the transformed $y$-parameterized equation, and the uniqueness of solutions is further proven by means of flow transformation techniques. Under the Gaussian assumption on $Y$, we derive both necessary and sufficient optimality conditions for the control problem. Through the $y$-parameterized transformation, the LQ optimal control problem with inside information is converted into an LQ problem with random coefficients. We introduce a stochastic Riccati equation to solve for the optimal feedback control, and a numerical example is provided to verify the impact of inside information $Y$.
	
	For future research, we will extend the framework to Stackelberg differential games under inside information, where the leader possessing advanced private information acts first, and the follower without information respond thereafter. We will publish related results elsewhere in the near future.

\section*{Appendix}

\setcounter{equation}{-1}
\renewcommand{\theequation}{A.\arabic{equation}} 

	We prove that \eqref{state} has at most one solution. Denote $\sigma(t, x) := \sigma (t, x, u(t,Y), Y)$, for given $u$ and $Y$. We need the following assumption.
	
	\textbf{(H3)}  $\sigma$ is continuously differentiable with respect to $t, u$, $\sigma_t$ is bounded by $L\left(1+|x|+\left|u\right|\right)$.  The mixed derivatives $\sigma_{xt}$, $\sigma_{xu}$, $u_t$ exist and $\sigma_{xt}$, $\sigma_{xu}$, $\sigma_u$, $u_t$  are bounded. 
	
	Define $F: [0,T]\times \mathbb{R} \times \mathbb{R} \rightarrow \mathbb{R}$ as the solution of
	\begin{align}
		\begin{cases}
			\frac{\partial F}{\partial r}(t, r, x)=\sigma(t, F(t, r, x)),\\
			F(t, 0, x)=x.
		\end{cases}
	\end{align}
	Under \textbf{(H1)}, theory of \textit{ordinary differential equations} (ODE) implies that $F$ defines a flow with $t$ as a parameter. Since $\sigma$ is of class $C^2$, for any fixed $t \in [0,T]$, $r \in \mathbb{R}$, the map $x \mapsto F(t, r, x)$ is a $C^2$-diffeomorphism on $\mathbb{R}$. Let $H(t, r, x)$ be the inverse mapping of $F$ with respect to the variable $x$, i.e.  $H(t, r, x)=F^{-1}(t, r, x)$, which is again of $C^2$ class. Therefore, the following relationship holds:
	\begin{align}
		F(t, r, H(t, r, x))=x, \quad H(t, r, F(t, r, x))=x, \quad \forall r \in \mathbb{R}, \quad x \in \mathbb{R}. \label{F-H}
	\end{align}
	We give some results about $F$ and its inverse $H$. Differentiating both sides of the first equation in \eqref{F-H} with respect to $r$ and $x$, respectively, we can obtain
	\begin{gather}
		\frac{\partial F}{\partial r}\big(t, r, H(t, r, x)\big)+\frac{\partial F}{\partial x}\big(t, r, H(t, r, x)\big) \frac{\partial H}{\partial r}(t, r, x)=0,\label{F_r}\\
		\frac{\partial F}{\partial x}\big(t, r, H(t, r, x)\big) \frac{\partial H}{\partial x}(t, r, x)=1,\label{F_x}
	\end{gather}
	and moreover
	\begin{align}
		\frac{\partial}{\partial r}\left(\frac{\partial F}{\partial x}(t, r, x)\right)&=\frac{\partial}{\partial x}\left(\frac{\partial}{\partial r} F(t, r, x)\right)=\frac{\partial}{\partial x}\big(\sigma(t, F(t, r, x))\big)\nonumber\\
		&=\sigma_x(t, F(t, r, x)) \frac{\partial F}{\partial x}(t, r, x),
	\end{align}
	which is an ODE with respect to $\frac{\partial F}{\partial x}(t, r, x)$. Its solution is 
	\begin{align}
		\frac{\partial F}{\partial x}(t, r, x)=\exp \left\{\int_0^r \sigma_x(t, F(t, s, x)) d s\right\}, \label{F_x_exp}
	\end{align}
	showing that $\dfrac{\partial F}{\partial x}$ is strictly positive and uniformly bounded. Combining \eqref{F_r} and \eqref{F_x}, we can derive that
	\begin{align}
		\frac{\partial H}{\partial r}(t, r, x)&=- \bigg(\frac{\partial F}{\partial x}\big(t, r, H(t, r, x)\big)\bigg)^{-1}\frac{\partial F}{\partial r}\big(t, r, H(t, r, x)\big)\nonumber \\
		&=-\frac{\partial H}{\partial x}(t, r, x) \sigma(t, F(t, r, H(t, r,x)))
		=-\frac{\partial H}{\partial x}(t, r, x) \sigma(t,x ).\label{H_r}
	\end{align}
	
	Let $X(t) = X(t, Y), t \in [0, T]$ be any solution of \eqref{state}, we will prove that this solution is unique.
	
	\noindent(i)  Define $Z(t) := H(t, B(t), X(t))$, then $X(t) = F(t, B(t), Z(t))$. Applying Russo-Vallois's formula, we obtain
	\begin{align*}
		Z(t) &= H(0, 0, X(0))+ \int_0^t \frac{\partial H}{\partial t}(s, B(s), X(s)) ds +\int_0^t \frac{\partial H}{\partial r}(s, B(s), X(s)) \, d^{-} B(s) \\
		&\qquad+ \int_0^t \frac{\partial H}{\partial x}(s, B(s), X(s)) \, d^{-} X(s) +\frac{1}{2} \int_0^t \Bigg\{ \frac{\partial^2 H}{\partial r^2}(s, B(s), X(s)) \\
		&\qquad+ 2\frac{\partial^2 H}{\partial r \partial x}(s, B(s), X(s)) \, \sigma\big(s, X(s)\big) 
		 + \frac{\partial^2 H}{\partial x^2}(s, B(s), X(s)) \, \big(\sigma\big(s, X(s)\big)\big)^2 \Bigg\} ds.
	\end{align*}
	Here, we use the facts that
	\begin{align*}
		d[B,B]_t=dt,\qquad
		d[B,X]_t=\sigma(t,X(t))dt,\qquad
		d[X,X]_t=\sigma^2(t,X(t))dt.
	\end{align*}
	
	Taking partial derivatives with respect to $x$ and $r$ in \eqref{H_r}, it yields
	\begin{align*}
		\frac{\partial^2 H}{\partial r \partial x}(t, r, x)&=- \frac{\partial^2 H}{\partial x^2}(t, r, x)\sigma(t, x) - \frac{\partial H}{\partial x}(t, r, x)\sigma_x(t, x),\\
		\frac{\partial^2 H}{\partial r^2}(t, r, x)&=-\frac{\partial^2 H}{\partial r \partial x}(t, r, x)\sigma(t, x)\\
		&= \frac{\partial^2 H}{\partial x^2}(t, r, x) \sigma^2(t, x)+\frac{\partial H}{\partial x}(t, r, x)(\sigma_x\sigma)(t, x).
	\end{align*}
	
	Noticing $H(0, 0, x_0) = H(0,0, F(0, 0, x_0)) = x_0$ and \eqref{H_r}, we have
	\begin{align*}
		Z(t) &= x_0 + \int_0^t \frac{\partial H}{\partial t}(s, B(s), X(s)) ds+ \int_0^t \frac{\partial H}{\partial r}(s, B(s), X(s))d^{-} B(s)  \\
		&\qquad+\int_0^t \frac{\partial H}{\partial x}(s, B(s), X(s)) \Big[b(s, X(s))dt + \sigma(s, X(s))d^{-}B(s)\Big]\\
		&\qquad+\frac{1}{2} \int_0^t \Bigg\{ \frac{\partial^2 H}{\partial r^2}(s, B(s), X(s)) + 2\frac{\partial^2 H}{\partial r \partial x}(s, B(s), X(s))\sigma\big(s, X(s)\big) \\
		&\qquad + \frac{\partial^2 H}{\partial x^2}(s, B(s), X(s)) \big(\sigma\big(s, X(s)\big)\big)^2 \Bigg\} ds \\
		&= x_0 +  \int_0^t \frac{\partial H}{\partial t}(s, B(s), X(s)) ds\\
		&\qquad+\int_0^t \frac{\partial H}{\partial x}(s, B(s), X(s))\Big[b(s, X(s))- \frac{1}{2}(\sigma_x\sigma)(s, X(s))\Big]dt,
	\end{align*}
	where $b(\cdot, X(\cdot)):=b(\cdot, X(\cdot), u(\cdot, Y), Y)$. Denote $G(t, r, x)=\frac{1}{F_x(t, r, x)}$, together with \eqref{F_x}, we get
	\begin{equation}
		\frac{\partial H}{\partial x}(t, r, x) = \frac{1}{F_x(t, r, H(t, r, x))}=G(t, r, H(t, r, x)).\label{G}
	\end{equation}
	Recalling the first equality in \eqref{F-H}, by differentiating both sides with respect to $t$, we obtain
	\begin{align*}
		\frac{\partial F}{\partial t}(t, r, H(t, r, x)) + \frac{\partial F}{\partial x}(t, r, H(t, r, x))\dfrac{\partial H}{\partial t}(t, r, x) = 0,
	\end{align*}
	thus
	\begin{align*}
		\frac{\partial H}{\partial t}(t, r, x) = -\dfrac{F_t\bigl(t, r, H(t, r, x)\bigr)}{F_x\bigl(t, r, H(t, r, x)\bigr)} =- G(t, r,  H(t, r, x))F_t(t, r, H(t, r, x)).
	\end{align*}
	Finally we get
	\begin{align}
		Z(t) = &\ x_0 +\int_0^t G(s, B(s), Z(s)) \Big[b(s, F(s, B(s),Z(s)))- F_t(s, B(s), Z(s))\nonumber\\
		&- \frac{1}{2}(\sigma_x\sigma)\big(s, F(s, B(s),Z(s))\big)\Big]dt.\label{Z}
	\end{align}
	
	\noindent (ii) We now prove the uniqueness of the solution to equation \eqref{Z}. Let $X_1$ and $X_2$ be two solutions to equation \eqref{state} with the same initial value $x_0$. For $i=1,2$, define $Z_i(t)=H(t,B(t),X_i(t))$, then $Z_1$ and $Z_2$ are two continuous solutions of equation \eqref{Z}. We have
	\begin{align*}
		Z_i(t)=x_0+\int_0^t \Psi(s,Z_i(s))ds,\quad i=1,2,
	\end{align*}
	where
	\begin{align*}
		\Psi(s,z) := 
			G(s,B(s),z) \Big[ b\big(s,F(s,B(s),z)\big) F_t(s,B(s),z)-\frac{1}{2}(\sigma_x\sigma) \big(s,F(s,B(s),z)\big)\Big].
	\end{align*}
	Fixing $\omega \in \Omega$, $B(\cdot,\omega)$ is continuous on $[0,T]$, it is bounded. Moreover,  $Z_1(\cdot)$ and $Z_2(\cdot)$ are continuous on $[0,T]$, they are also bounded. Hence there exist constants $R >0$ such that
	\begin{align*}
		|B(t)| + |Z_1(t)| +|Z_2(t)|\leq R,\quad t\in[0,T].
	\end{align*}
	Here we only require $\Psi$ to satisfy a generalized Lipschitz condition, i.e., for all $|z_1|,|z_2|\leq R$, there exists a  $L(\cdot) \in L^1(0, T)$, such that
	\begin{align*}
		|\Psi(t,z_1)-\Psi(t,z_2)| \leq L(t) |z_1-z_2|,	\quad t\in[0,T].
	\end{align*}
	Therefore, for $t\in[0,T]$, we have
	\begin{align*}
		|Z_1(t)-Z_2(t)| &\leq \int_0^t |\Psi(s,Z_1(s))-\Psi(s,Z_2(s))|ds \leq  \int_0^t L(s) |Z_1(s)-Z_2(s)|ds.
	\end{align*}
	By the following Lemma \ref{lemma in appendix}, we obtain $|Z_1(t)-Z_2(t)|=0$, $t\in[0,T]$. Thus, along each trajectory, 
	\begin{align*}
		Z_1(t)=Z_2(t),\quad t\in[0,T],
	\end{align*}
	which proves the uniqueness of the solution to equation \eqref{Z}. Since mapping $x \mapsto F(t, r, x)$ is $C^2$-diffeomorphism with respect to $x$ and $X_i(t) = F(t, B(t), Z_i(t))$, the solution to equation \eqref{state} is pathwise unique.
	
	\begin{mylem}\label{lemma in appendix}
		Let $\phi(\cdot) := |Z_1(\cdot) - Z_2(\cdot)| \geq 0 $ be a nonnegative continuous function, and let $L(\cdot)\in L^1(0,T)$ with $L(t) \geq 0$ a.e. $t \in [0,T]$. Suppose that
		\begin{align}
			\phi(t) \leq \int_{0}^{t} L(s)\phi(s)\,ds,\quad t \in [0,T].
		\end{align}
		Then $\phi(t) =0,\quad t \in [0,T]$.
	\end{mylem}
	
	\begin{proof}
		Define $\Phi(t) := \int_{0}^{t} L(s)\phi(s)\,ds$. Since $L \in L^1(0,T)$ and $\phi$ is continuous on $[0,T]$, we have $L(\cdot)\phi(\cdot)\in L^1(0,T)$. Hence $\Phi$ is continuous, and $\Phi'(t) = L(t)\phi(t),\quad \text{a.e. } t \in [0,T]$. 
		By the assumption $\phi(t) \leq \Phi(t)$, it follows that, for a.e. $t \in [0,T]$, $\Phi'(t) = L(t)\phi(t) \leq L(t)\Phi(t)$.
		
		Define
		\begin{align*}
			\Gamma(t) := \Phi(t) \exp\left\{ -\int_{0}^{t} L(s)\,ds \right\}.
		\end{align*}
		Since $\Phi$ and $t \mapsto \exp\left\{ -\int_{0}^{t} L(s)ds \right\}$ are continuous, we differentiate $\Gamma$ with respect to $t \in [0,T]$, to obtain
		\begin{align*}
			\Gamma'(t) 
			&= \Phi'(t)\exp\left\{ -\int_{0}^{t} L(s)\,ds \right\} 
			- \Phi(t)L(t)\exp\left\{ -\int_{0}^{t} L(s)\,ds \right\} \\
			&= \exp\left\{ -\int_{0}^{t} L(s)\,ds \right\} 
			\Big( \Phi'(t) - L(t)\Phi(t) \Big).
		\end{align*}
		Since $\Phi'(t) \leq L(t)\Phi(t)$, we have $\Gamma'(t) \leq 0,\ \text{a.e. } t \in [0,T]$. Therefore, $\Gamma$ is nonincreasing. Since $\Gamma(0) = \Phi(0) = 0$, for any $t \in [0,T]$, $\Gamma(t) \leq \Gamma(0) = 0$. 
		On the other hand, since $L \geq 0$ and $\phi \geq 0$, we have
		\begin{align*}
			\Gamma(t) = \Phi(t) \exp\left\{ -\int_{0}^{t} L(s)\,ds \right\} \geq 0.
		\end{align*}
		Hence $\Gamma(t) = 0,\ t \in [0,T]$, and $\Phi(t) = 0,\ t \in [0,T]$.
		Finally, by the assumption, $0 \leq \phi(t) \leq \Phi(t) = 0$. Therefore, $\phi(t) = 0,\ t \in [0,T]$. The proof is complete.    
	\end{proof}
	
	\noindent(iii) Fix $\omega\in\Omega$, we now verify the generalized Lipschitz continuity of $\Psi$ with respect to $z$. 
	
	We write $\Psi(s,z)=A(s,z)M(s,z)$, and denote
	\begin{align*}
		A(s,z) &:= G(s,B(s),z),\\
		M(s,z) &:= b\big(s,F(s,B(s),z)\big) - F_t(s,B(s),z) -\frac{1}{2}(\sigma_x\sigma) \big(s,F(s,B(s),z)\big).
	\end{align*}
	Then
	\begin{align}
		\partial_z\Psi(s,z) = A_z(s,z)M(s,z)+A(s,z)M_z(s,z).\label{Psi_z}
	\end{align}

	Firstly, we claim that $A$ and $A_z$ are bounded on compact sets. Noticing \eqref{F_x} and \eqref{G}, for $|r| \leq M, |\sigma_x| \leq L$, we have $ e^{-LM} \leq |F_x(t, r, z)| \leq e^{LM}$, then $A(s, z)$ is bounded for $s \in[0, T]$, $|z|\leq R$. We have
	\begin{align*}
		A_z(t,z) =\left. -\frac{F_{xx}(t,r,z)}{(F_x(t,r,z))^2}\right|_{r=B(s)}.
	\end{align*}
	Differentiating \eqref{F_x} with respect to $z$, we obtain
	\begin{align*}
		F_{xx}(t,r,z) = F_x(t,r,z) \int_{0}^{r} \sigma_{xx}\big(t,F(t,s,z)\big)F_x(t,s,z)\,ds.
	\end{align*}
	By the boundedness of $\sigma_{xx}$, we may assume that$ |\sigma_{xx}|\leq L$. Then, for $|r|\leq M$ and $|z|\leq R$,
	\begin{align*}
		|F_{xx}(t,r,z)| &\leq |F_x(t,r,z)| \int_{0}^{r} \left|\sigma_{xx}\big(t,F(t,s,z)\big)\right| \left|F_x(t,s,z\right)|ds  \\
		&\leq \exp\{LM\} \int_{0}^{|r|} L\exp\{LM\} ds \leq LM\exp\{2LM\}.
	\end{align*}
	Since $F_x$ is bounded away from zero,  $A_z$ is bounded on compact set.
	
	We next prove that $ M$ and $ M_z $ can be controlled by $u(t)$. Recalling the definition of $F$,
	\begin{align*}
		F(t,r,z) = z+\int_0^r \sigma\big(t,F(t,\rho,z),u(t,Y),Y\big)d\rho.
	\end{align*}
	Using the linear growth condition of $\sigma$, for $|r|\leq M$ and $|Z|\leq R$,  by Gronwall's inequality, we obtain
	\begin{align}
		|F(t,r,z)| \leq C \big(1+|u(t,Y)|\big). \label{F_u}
	\end{align}
	
	Moreover, by differentiating the equation with respect to $t$, we have
	\begin{align*}
		F_t(t,r,z) &= \int_0^r \Big[ \sigma_t\big(t,F(t,\rho,z),u(t,Y),Y\big) + \sigma_x\big(t,F(t,\rho,z),u(t,Y),Y\big)F_t(t,\rho,z)\\
		&\qquad + \sigma_u\big(t,F(t,\rho,z),u(t,Y),Y\big)u_t(t,Y) \Big]d\rho .
	\end{align*}
	Together with (H3) and \eqref{F_u}, we get
	\begin{align*}
		|F_t(t,r,z)| &\leq \int_0^{|r|} C\big(1+|F(t,\rho,z)|+|u(t,Y)|\big)d\rho +L\int_0^{|r|}|F_t(t,\rho,z)|d\rho \\
		&\leq C\big(1+|u(t,Y)|\big) +L\int_0^{|r|}|F_t(t,\rho,z)|d\rho.
	\end{align*}
	Again by Gronwall's inequality,
	\begin{align}
		|F_t(t,r,z)| \leq C\big(1+|u(t,Y)|\big),
		\qquad |r|\leq M,\ |z|\leq R.
		\label{Ft_u}
	\end{align}
	Particularly, 
	\begin{align}
		|F_t(s,B(s),z)| \leq C\big(1+|u(s,Y)|\big), \qquad s\in[0,T],\ |z|\leq R.
		\label{Ft_B_u}
	\end{align}
	
	By the linear growth condition of $b$ and $\sigma$, the boundedness of $\sigma_x$, and \eqref{F_u}, \eqref{Ft_B_u}, we obtain
	\begin{align*}
		|M(s,z)| &\leq \big|b\big(s,F(s,B(s),z),u(s,Y),Y\big)\big| + \big|F_t(s,B(s),z)\big| \\
		&\quad +\frac{1}{2} \big|(\sigma_x\sigma)\big(s,F(s,B(s),z),u(s,Y),Y\big)\big| \\
		&\leq C\big(1+|u(s,Y)|\big), \qquad s\in[0,T], |z|\leq R.
	\end{align*}
	
	Then we estimate $M_z$. Differentiating $M$ with respect to $z$, we have
	\begin{align*}
		M_z(s,z) &= b_x\big(s,F(s,B(s),z),u(s,Y),Y\big)F_x(s,B(s),z) - F_{tx}(s,B(s),z) \\
		&\quad -\frac{1}{2} \Big( \sigma_{xx}\sigma+\sigma_x^2 \Big)\big(s,F(s,B(s),z),u(s,Y),Y\big)F_x(s,B(s),z).
	\end{align*}
	To estimate $F_{tx}$, we differentiate the formula of $F_x$ with respect to $t$, 
	\begin{align*}
		F_{tx}(t,r,z) &= F_x(t,r,z) \int_0^r \Big[ 	\sigma_{tx}\big(t,F(t,\rho,z),u(t,Y),Y\big) + \sigma_{xx}\big(t,F(t,\rho,z),u(t,Y),Y\big)F_t(t,\rho,z)\nonumber\\
		&\qquad+ \sigma_{xu}\big(t,F(t,\rho,z),u(t,Y),Y\big)u_t(t,Y) \Big]d\rho.
	\end{align*}
	Under (H3), and the estimations
	\begin{align*}
		|F(t,r,z)| \leq C\big(1+|u(t,Y)|\big),\quad
		|F_t(t,r,z)| \leq C\big(1+|u(t,Y)|\big), \quad
		|F_x(t,r,z)| \leq C,
	\end{align*}
	we could derive
	\begin{align*}
		|F_{tx}(t,r,z)| 
		\leq |F_x(t,r,z)| \int_0^{|r|}C_{\omega,K}\big(1+|u(t,Y)|\big)d\rho 
		\leq C\big(1+|u(t,Y)|\big).
	\end{align*}
	Hence $|F_{tx}(s,B(s),z)| \leq C\big(1+|u(s,Y)|\big), s\in[0,T],\ z\in K.$
	
	Finally, by the boundedness of $b_x$, $\sigma_x$, $\sigma_{xx}$, together with the linear growth of $\sigma$, estimations of $F$, $F_x$, $F_{tx}$, we obtain
	\begin{align*}
		|M_z(s,z)| &\leq \big|b_x\big(s,F(s,B(s),z),u(s,Y),Y\big)F_x(s,B(s),z)\big| + \big|F_{tx}(s,B(s),z)\big| \\
		&\quad + \frac{1}{2} \left| \Big( \sigma_{xx}\sigma+\sigma_x^2 \Big)\big(s,F(s,B(s),z),u(s,Y),Y\big) F_x(s,B(s),z) \right| \\
		&\leq C\big(1+|u(s,Y)|\big),
	\end{align*}
	for all $s\in[0,T]$ and $|z|\leq R$. 
	
	Based on the above discussion and \eqref{Psi_z}, there exists a constant $C>0$ such that
	\begin{align*}
		\left|\Psi_z(s, z)\right| \leq C(1 + \left|u(s, Y)\right|) := L(s) \in L^1(0, T),
	\end{align*}
	for all $s\in[0,T]$ and $|z|\leq R$, which means for all $|z_1|,|z_2|\leq R$, there exists $L(\cdot) \in L^1(0,T)$ such that
	\begin{align*}
		|\Psi(s,z_1)-\Psi(s,z_2)| \leq L(s)|z_1-z_2|, \qquad s\in[0,T].
	\end{align*}
    We complete the proof of uniqueness.
\end{document}